\documentclass[12pt,a4paper,reqno]{amsart}

\usepackage[top=3.5cm,bottom=3cm,left=2.5cm,right=2.5cm]{geometry}
\usepackage[T1]{fontenc}
\usepackage[utf8]{inputenc}
\usepackage[english]{babel}
\usepackage{float}
\usepackage{xspace}
\usepackage{lmodern}
\DeclareFontFamily{U}{wncyr}{}
\DeclareFontShape{U}{wncyr}{m}{n}{<->wncyr10}{}
\DeclareFontShape{U}{wncyr}{m}{it}{<->wncyi10}{}
\DeclareFontShape{U}{wncyr}{m}{sc}{<->wncysc10}{}
\DeclareFontShape{U}{wncyr}{b}{n}{<->wncyb10}{}
\DeclareTextCommand{\guillemotleft}{T1}{%
  {\fontencoding{U}\fontfamily{wncyr}\selectfont\symbol{"3C}}%
}
\DeclareTextCommand{\guillemotright}{T1}{%
  {\fontencoding{U}\fontfamily{wncyr}\selectfont\symbol{"3E}}%
}

\usepackage{amsmath,amsfonts,amsthm,amssymb}
\usepackage{tikz-cd}

\usepackage{graphicx}
\usepackage{tikz}
\usepackage{ifthen}
\usetikzlibrary{arrows,calc,intersections,patterns}
\usepackage{rotating}

\usepackage{caption}
\usepackage{subcaption}
\usepackage{enumitem}
\usepackage{cite}

\allowdisplaybreaks[3] 
\usepackage{tikz-cd}
\usepackage{mathtools}
\usepackage{hyperref}
\usepackage{xcolor}

\numberwithin{equation}{section}
\DeclareRobustCommand{\SkipTocEntry}[5]{}

\renewcommand{\d}{\,\mathrm{d}}

\def\p{\partial}

\def\d{\,\mathrm{d}}

\usepackage[]{hyperref}
\hypersetup{
    colorlinks=true,       
    linkcolor=blue,         
    citecolor=magenta,       
    filecolor=magenta,     
    urlcolor=cyan          
}
\usepackage[nameinlink,noabbrev,capitalise]{cleveref}

\newtheorem{theorem}{Theorem}[section]
\newtheorem{proposition}[theorem]{Proposition}
\newtheorem{corollary}[theorem]{Corollary}
\newtheorem{lemma}[theorem]{Lemma}

\newtheorem{theorembis}{Theorem}

\theoremstyle{definition}

\theoremstyle{remark}
\newtheorem{remark}[theorem]{Remark}
\crefname{lemma}{Lemma}{Lemmas}
\crefname{proposition}{Proposition}{Propositions}

\begin{document}

\title[Strichartz Estimates for Baouendi--Grushin Operators]{Strichartz Estimates for a class of Baouendi--Grushin Operators}

\author{Nicolas Burq}
\address{Université Paris-Saclay, CNRS,  Laboratoire de mathématiques d’Orsay 91405 Orsay, France,  and Institut Universitaire de France}
\email{nicolas.burq@universite-paris-saclay.fr}
\author{Mickaël Latocca}
\address{Laboratoire de Mathématiques et de Modélisation d’Évry (LaMME), Université d’Évry, 23 Bd François Mitterrand, 91000 Évry-Courcouronnes, France}
\email{mickael.latocca@univ-evry.fr}

\thanks{
\textit{Acknowledgements}. This research was supported by the European research Council (ERC) under the
European Union’s Horizon 2020 research and innovation programme (Grant agreement
101097172 - GEOEDP).
}

\date{\today}

\begin{abstract}
\begin{sloppypar}
We prove Strichartz estimates for a class of Baouendi--Grushin operators acting either on the Euclidean space or a product of the type $\mathbb{R}^{d_1} \times M$, where $(M,g)$ is a smooth compact manifold with no boundary. We then give an application of these Strichartz estimates to the Cauchy theory for the associated Schrödinger equations.  
\end{sloppypar}
\end{abstract}

\maketitle

\section{Introduction}

\subsection{The Baouendi--Grushin operator} We consider a degenerate elliptic operator introduced by Baouendi and Grushin \cite{Baouendi67,Grushin71}, that we write $-\Delta _G$ and which is defined by  
\[
    -\Delta_G = - \Delta_x^2 - |x|^2\Delta_y^2, \quad \text{ for } (x,y) \in \mathbb{R}^{d_1} \times M,  
\]
where $(M,g)$ is either the Euclidean space $\mathbb{R}^{d_2}$ or a smooth compact manifold with no boundary.

We also define the associated Sobolev spaces $H^s_G=\{u \in L^2(\mathbb{R}^{d_1} \times M): \|u\|_{H^s_G} < \infty\}$ with   
\[
    \|u\|_{H^s_G}^2 = \|(\operatorname{id}-\Delta_G)^{\frac{s}{2}}\|_{L^2}^2.    
\]
Our original motivation is the study of the Cauchy problem for the following nonlinear Schrödinger equations: 
\begin{equation}
    \label{eq.NLS-G}
    \tag{NLS-G}
    \begin{aligned}
        i\partial_t u + \Delta_G u &= F(u) := |u|^{\kappa - 1}u \\
        u(0)&=u_0 \in H^s_G. 
    \end{aligned}
\end{equation}

In view of the Sobolev embedding $H^{s}_G \hookrightarrow L^{\infty}$, which holds as soon as $s>\frac{d_1+2d_2}{2}$ (see \cref{thm.sobolev}), it follows that \eqref{eq.NLS-G} is locally well-posed in $\mathcal{C}^0([0,T],H^{s}_G)$ for such values of $s$. We refer to this theory as the \textit{classical} Cauchy theory. 

The non-dispersive nature of \eqref{eq.NLS-G} in the case $(d_1,d_2)=(1,1)$ was observed in \cite{BGX}, so that up to the knowledge of the authors, no better Cauchy theory is known than the classical one. 

Our starting point is that as soon as $d_2 \geqslant 2$ one expects an improvement on the classical Cauchy theory by using dispersion on the $y$ direction by elementary means. 

\subsection{Main results}

Our main result is a collection of Strichartz estimates associated to the Schrödinger propagator $e^{it\Delta_G}$ and also for its fractional counterpart $e^{it(-\Delta_G)^{\sigma}}$. 

We say that a triple $(p,q,r) \in [2,\infty]^3$ is \textit{admissible} if the following scaling relation hold: 
\begin{equation}
    \label{eq.adm-S}
    \frac{2}{p} + \frac{d_1}{q} + \frac{2d_2}{r} = \frac{d_1+2d_2}{2} - \gamma_{q,r}, \text{ where}  
\end{equation}
\begin{equation}
    \label{eq.adm-S2}
    \gamma_{q,r} = (d_2+1)\left(\frac{1}{2} - \frac{1}{r}\right) + d_1\left(\frac{1}{2} - \frac{1}{q}\right), 
\end{equation}
and $(p,r)$ are restricted by 
\begin{equation}
    \label{eq.adm-S3}
    \frac{1}{d_2-1} > \frac{1}{2} - \frac{1}{r} > \begin{cases}
        \frac{1}{2d_2} & \text{ in the Euclidean case,} \\ \frac{1}{2d_2} + \frac{1}{d_2p} &\text{ in the compact manifold case}. 
    \end{cases} 
\end{equation}

Similarly, when $\sigma > 1$, a triple $(p,q,r) \in [2,\infty]^3$ is said to be $\sigma$-admissible if 
\begin{equation}
    \label{eq.adm-fS}
    \frac{2\sigma}{p} + \frac{d_1}{q} + \frac{2d_2}{r} = \frac{d_1+2d_2}{2} - \gamma_{q,r}, \text{ where}  
\end{equation}
\begin{equation}
    \label{eq.adm-fS2}
    \gamma_{q,r} = 
    d_2(2-\sigma)\left(\frac{1}{2} - \frac{1}{r}\right) + d_1\left(\frac{1}{2} - \frac{1}{q}\right),
\end{equation}
and $(p,r)$ are restricted by 
\begin{equation}
    \label{eq.adm-fS3}
    \frac{1}{d_2} > \frac{1}{2} - \frac{1}{r} > \begin{cases}
        \frac{1}{2d_2} & \text{ in the Euclidean case,} \\ \frac{1}{2d_2} + \frac{1}{d_2p} &\text{ in the compact manifold case}. 
    \end{cases} 
\end{equation}

Our main result is the following. 

\begin{theorembis}[Strichartz estimates for Schrödinger-Baouendi--Grushin propagators]\label{thm.main-Strichartz} \quad 
\begin{enumerate}[label=(\roman*)]
    \item Let $d_1 \geqslant 1$ and $d_2 \geqslant 2$. Let $(p,q,r)$ and $(a,b,c)$ be admissible triples such that $(p,a)\neq (2,2)$. Let $\varepsilon \ll 1$. Then for all $u \in H_G^{\gamma_{q,r}}(\mathbb{R}^{d_1}\times \mathbb{R}^{d_2})$ (or $H_G^{\gamma_{q,r}}(\mathbb{R}^{d_1}\times M)$ where $(M,g)$ is a smooth compact manifold with no boundary), and for all $G \in L^{p'}_TL^{q'}_xL^{r'}_y$, $F\in L^{a'}_TL^{b'}_xL^{c'}_y$, the following estimates hold: 
    \begin{equation}
        \label{eq.grushin-strichartz}
        \|e^{it\Delta_G}u\|_{L^{p}_TL^q_xL^r_y} \leqslant C(p,q,r) \|u\|_{H^{\gamma_{q,r}+\varepsilon}_G},
    \end{equation}
    \begin{equation}
        \label{eq.grushin-strichartz-dual}
        \left\|\int e^{-it'\Delta_G}G(t')\mathrm{d}t'\right\|_{L^2_{x,y}} \leqslant C \|(\operatorname{id}-\Delta_G)^{\frac{\gamma_{q,r}+\varepsilon}{2}}G\|_{L^{p'}_TL^{q'}_xL^{r'}_y},
    \end{equation}
    \begin{equation}
        \label{eq.grushin-strichartz-inhomogeneous}
        \left\|\int_0^t e^{i(t-t')\Delta_G}F(t')\mathrm{d}t'\right\|_{L^p_TL^q_xL^r_{y}} \leqslant C \|(\operatorname{id}-\Delta_G)^{\frac{\gamma_{q,r}+\gamma_{b,c}+\varepsilon}{2}}F\|_{L^{a'}_TL^{b'}_xL^{c'}_y}.
    \end{equation} 
    \item Let $d_1 \geqslant 1$, $d_2 \geqslant 1$ and $\sigma \in (1,2)$. Let $(p,q,r)$ and $(a,b,c)$ be $\sigma$-admissible triples such that $(p,a)\neq (2,2)$. Then \eqref{eq.grushin-strichartz}, \eqref{eq.grushin-strichartz-dual} and \eqref{eq.grushin-strichartz-inhomogeneous} hold with $\Delta_G$ replaced with $(-\Delta_G)^{\sigma}$ and where $\gamma_{q,r}$ stands for \eqref{eq.adm-fS2}. In the compact case, $\gamma_{q,r}$ should be replaced by $\gamma_{q,r} + \frac{2(\sigma - 1)}{p}$. 
\end{enumerate}
\end{theorembis}

\begin{remark}[Use of the Christ--Kiselev Lemma]
Since the inhomogeneous estimate \eqref{eq.grushin-strichartz-inhomogeneous} is classically obtained from the homogeneous one~\eqref{eq.grushin-strichartz}, and its dual \eqref{eq.grushin-strichartz-dual} by an application of Christ--Kiselev's Lemma, we should impose $(p,a)\neq (2,2)$. However, up to losing an extra $\varepsilon$ of regularity we can relax this assumption by applying Sobolev's embedding in $t$ and use the fact that $i\p_t e^{it\Delta_G}F=-\Delta_GF$.
\end{remark}

\begin{remark}[Comparison with the Sobolev embedding] As a comparison, an application of the Sobolev embedding yields \eqref{eq.grushin-strichartz} with
\[
    \gamma_{q,r}:=\gamma_{\text{Sob}} = 2d_2\left(\frac{1}{2}-\frac{1}{r}\right) + d_1\left(\frac{1}{2} - \frac{1}{q}\right).
\]
Writing $\gamma_{\text{Stri}}$ for \eqref{eq.adm-S2} (resp. \eqref{eq.adm-fS2}) we have 
\[
    \gamma_{\text{Sob}}  - \gamma_{\text{Stri}}= \begin{cases}
        (d_2-1)\left(\frac{1}{2} - \frac{1}{r}\right) & \text{ in the case } \sigma = 1,\\ 
        d_2\sigma \left(\frac{1}{2} - \frac{1}{r}\right)& \text{ in the case } \sigma \neq 1,
    \end{cases}
\]
which is positive as soon as $r>2$, therefore \cref{thm.main-Strichartz} is an improvement on the Sobolev embedding, as expected. 

Note however that this gain only happens on the $y$ variable (it is in $y$ that we use the dispersion), whereas in $x$ we only perform Sobolev's embedding, therefore not improving the estimate in $x$. 
\end{remark}

\begin{remark}[On scaling and optimality] Condition \eqref{eq.adm-S} (resp. \eqref{eq.adm-fS}) can be seen from scaling. Indeed, plugging $u=u_{\lambda}$ in \eqref{eq.grushin-strichartz}, where $u_{\lambda}(t,x,y)=u(\lambda^2t,\lambda x,\lambda ^2y)$ (resp. $u_{\lambda}(t,x,y)=u(\lambda^{2\sigma}t,\lambda x,\lambda ^2y)$) one can see that \eqref{eq.grushin-strichartz} is invariant under such scaling provided \eqref{eq.adm-S} (resp. \eqref{eq.adm-fS}) holds. 

The upper bound of the restriction on the admissible values of $\frac{1}{2}-\frac{1}{r}$ given by \eqref{eq.adm-S3} (resp. \eqref{eq.adm-fS3}) stems from an application of the Hardy-Littlewood-Sobolev inequality (see the proof of \cref{prop.strichartz-modewise}). 
On the other hand, the lower bound on $\frac{1}{2}-\frac{1}{r}$ given by \eqref{eq.adm-S3} (resp. \eqref{eq.adm-fS3}) is required in our argument in order to ensure summability of mode-wise Strichartz estimates in order to infer frequency-localized estimates.

The $\varepsilon$ loss in \cref{thm.main-Strichartz} stems for an interpolation argument in \cref{coro.disp-mode-m} and an application of the Cauchy-Schwarz inequality in \cref{eq.CSapp}. We believe that a more subtle interpolation argument as well as the use of a Littlewood-Paley inequality adapted to the Baouendi--Grushin operator might be used to avoid this loss.  
\end{remark}

\begin{remark}[On the nondispersive case $d_2=1$] When $d_2=1$, the propagator $e^{it\Delta_G}$ is non-dispersive, which is consistent with the fact that $d_2=1$ is excluded from \cref{thm.main-Strichartz}. This is not surprising, as we ultimately rely on dispersion of the half-wave equation in the $y$ variable, which is non-dispersive in dimension $1$. This fact was observed in \cite{GG10} where the authors constructed an explicit example of non-dispersive $u(t)=e^{it\Delta_G}u_0$ by choosing $u_0$ localized in $y$ frequencies and supported by one Hermite mode, namely: 
    \[
        u_0(x,y)=N^{-\frac{1}{2}}\int_0^{\infty} e^{iy\eta - \frac{1}{2}\eta x^2 - \frac{\eta}{N^2}}\d\eta, 
    \]
which is such that $u(t,x,y)=u_0(x,y-t)$, therefore is non-dispersive as Lebesgue norms are conserved. We refer to the introduction of \cite{GG10} for more details. 
\end{remark}

Next, as an illustration of our dispersive estiamtes, we use these Strichartz estimates to improve on the classical local Cauchy theory for \eqref{eq.NLS-G}, and also  its fractional counterpart: 
\begin{equation}
    \label{eq.fNLS-G}
    \tag{FNLS-G}
    \begin{aligned}
        i\partial_tu +(-\Delta_G)^{\sigma} u &= |u|^{\kappa -1}u \\ 
        u(0) & = u_0\in H^s_G.  
    \end{aligned}
\end{equation}

For simplicity, we restrict our attention to $(d_1,d_2)=(1,2)$ and $\sigma \in [1,2)$. Let us introduce a space adapted to the Strichartz estimates: 
\[
    X^s_T := \displaystyle\bigcap_{(p,q,r) \;\;\sigma-\text{admissible}} L^p_T(-\Delta_G )^{-\frac{s-\gamma_{q,r}-\varepsilon}{2}}(L^q_xL^r_y).
\] 

\begin{theorembis}[Local Cauchy theory for \eqref{eq.NLS-G} and \eqref{eq.fNLS-G}]\label{thm.main-Cauchy} \quad 
\begin{enumerate}[label=(\roman*)]
    \item Assume $\kappa = 5$. Let $s>2$ and $u_0 \in H^s_G(\mathbb{R}_x\times \mathbb{R}^{2}_y)$. There exists $T=T(\|u_0\|_{H^s})>0$ and a unique solution $u \in X^s_T \cap \mathcal{C}^0([0,T],H^s_G)$ to \eqref{eq.NLS-G}.
    If $M$ is a smooth compact manifold of dimension $2$ with no boundary, then the same result holds for $u_0 \in H^s_G(\mathbb{R}\times M)$ and $s>2$.
    \item Assume $\kappa = 3$, $\sigma \in (1,2)$. Let $s>\frac{5}{2}-\sigma$, and $u_0 \in H^s_G(\mathbb{R}_x\times \mathbb{R}^{2}_y)$. Then there exists $T=T(\|u_0\|_{H^s})>0$ and a unique solution $u\in X^s_T\cap \mathcal{C}^0([0,T],H^s_G)$ to \eqref{eq.fNLS-G}. In particular when $s=\sigma \in (\frac{5}{4},2)$, these solutions are global in time.
\end{enumerate}
\end{theorembis}

\begin{remark} The purpose of \cref{thm.main-Cauchy} is only to exemplify that using standard methods, Strichartz estimates obtained in \cref{thm.main-Strichartz} imply corresponding improvements on the classical local Cauchy theory for \eqref{eq.NLS-G} (resp. \eqref{eq.fNLS-G}). In order to avoid lengthening the paper, we have restricted our attention to a few values of the involved parameters $s$, $\sigma$, $\kappa$. In particular, we deal with integer power nonlinearity, but this is only a technical restriction, which could be overcome: see for instance \cite{BG}, in which the authors develop frequency analysis tools adapted to the study of \eqref{eq.NLW-H}, similar treatment could be conducted in our case.  
\end{remark}

\subsection{Some related results} 

When $d_2=1$, as  explained in \cite{GG10}, the Schrödinger equation for the Baouendi--Grushin operator is non-dispersive. However, it is possible to obtain \cite[Theorem 1.6]{GMS23} anisotropic estimates of the form 
\begin{equation}
    \label{eq.restriction-Strichartz}
    \|e^{it\Delta_G}\|_{L^{\infty}_yL^p_tL^q_x} \lesssim \|u\|_{L^2_{x,y}},
\end{equation}
which hold whenever $2< q \leqslant p\leqslant\infty$ and $\frac{2}{p} + \frac{d_1}{q} = \frac{d_1+2}{2}$. The estimate \eqref{eq.restriction-Strichartz} is very different from estimates of the type \eqref{eq.grushin-strichartz} as integration on time and in the $y$ variable have been switched. Such estimates are typically obtained by using Fourier restriction methods based on \cite{Muller90} (see \cite{LS16} for an adaptation to the Baouendi--Grushin setting), and were obtained for the Heisenberg operator in \cite{BBG}. 

When dispersion is expected, which is our case for the Schrödinger equation as soon as $d_2 \geqslant 2$, a usual strategy to obtain Strichartz estimates is to prove dispersion estimates and apply a $TT^*$ argument. Such strategy was implemented for the free solutions to the wave equation posed on the Heisenberg group $\mathbb{H}^n$, 
\begin{equation}
    \label{eq.NLW-H}
    \tag{NLW-$\mathbb{H}^n$}
    \begin{aligned}
        \partial_t^2u - \Delta_{\mathbb{H}^1} u &= F(u) \\ 
        (u(0), \partial_t u(0)) & = (u_0,u_1) \in H^s \times H^{s-1}(\mathbb{H}^n).  
    \end{aligned}
\end{equation}
In terms of scaling this corresponds to $(d_1,d_2)=(2,1)$ for our Baouendi--Grushin setting. Dispersion and Strichartz estimates were established in \cite{BGX} and used in \cite{BG} to derive an improvement on the classical local Cauchy theory for \eqref{eq.NLW-H} when $F(u)$ is a derivative nonlinearity. 

\subsection{Structure of the article and notation}

The main result, \cref{thm.main-Strichartz} is proven in \cref{sec.stri} first in the Euclidean case, and then for compact manifolds. Some useful results are recalled in \cref{sec.prelim}. 
Finally, the Cauchy theory of \cref{thm.main-Cauchy} is obtained in \cref{sec.cauchy}.

In the following, we make use of common notation. In particular we will write $\langle z \rangle =\sqrt{1+|z|^2}$.

If $a \in [1, \infty]$ we write $a'\in [1,\infty]$ its Hölder conjugate, that is $\frac{1}{a'} + \frac{1}{a} = 1$. 

We write $L^pX$ as a shorthand for $L^p((0,T),X)$. 

\section{Some useful results}\label{sec.prelim} 

\subsection{Commutators and embeddings}

We recall that $[\p_{x_i},x_i\p_{y_j}]=\p_{y_j}$, which expresses the fact that a derivative in the $y$ direction can be obtained by a combination two order one derivatives associated to $-\Delta_G$. This fact is responsible for the following Sobolev embedding that we recall. 

\begin{theorem}[See \cite{FS74}]\label{thm.sobolev}
Let $s>\frac{d_1+2d_2}{2}$. Then the following embedding $H^{s}_G \hookrightarrow L^{\infty}$ holds. 
\end{theorem}

\begin{corollary} Let $s>\frac{d_1+2d_2}{2}$ and $p\in [2,\infty]$. Then, for all $u \in W^{\frac{2s}{p},p}_G$ there holds 
\begin{equation}
    \label{eq.sobolev-Wkp}
    \|u\|_{L^{\infty}} \lesssim \|(\operatorname{id}-\Delta_G)^{\frac{s}{p}}u\|_{L^p}. 
\end{equation}
\end{corollary}

\begin{proof} This follows using Stein's complex interpolation theorem \cite{S56} between the Sobolev embedding  
\[
    \|(\operatorname{id}-\Delta_G)^{-\frac{s}{2}}u\|_{L^{\infty}_{x,y}} \lesssim \|u\|_{L^2_{x,y}}    
\]
and $\|u\|_{L^{\infty}_{x,y}} = \|u\|_{L^{\infty}_{x,y}}$.
\end{proof}

\begin{lemma}\label{lem.x-control}
For any $s\in [0,1]$, the following embedding $H^s_G \hookrightarrow L^2_yH^s_x$ is continuous.
\end{lemma}

\begin{proof} First, observe for any smooth function $u$, we can integrate by parts to compute 
\begin{align*}
    \|\nabla_x u\|_{L^2_{x,y}}^2 &\leqslant \int_{\mathbb{R}^{d_1}\times\mathbb{R}^{d_2}}(|\nabla_x u|^2 + |x|^2|\nabla_yu|^2) \mathrm{d}x\mathrm{d}y = (-\Delta_G u,u)_{L^2_{x,y}} = \|u\|^2_{H^1_G},
\end{align*}
which shows $H^1_G \hookrightarrow L^2_{y}H^1_x$, i.e. the case $s=1$ of the lemma. The claimed result now follows from complex interpolation with the $s=0$ case.
\end{proof}

\subsection{A frequency decomposition adapted to the Baouendi--Grushin operator}

We outline a frequency decomposition for the Baouendi--Grushin operator acting on $\mathbb{R}^{d_1} \times \mathbb{R}^{d_2}$. In the case $\mathbb{R}^{d_1} \times M$ for smooth boundaryless compact manifolds $M$, one would simply follow the same arguments, that is: using a spectral decomposition of the Laplace operator on $M$ and expanding the Baouendi--Grushin operator along Hermite modes. In the following, we only detail the Euclidean case.  

Let $u \in \mathcal{S}(\mathbb{R}^{d_1}\times\mathbb{R}^2)$ and take the partial Fourier transform in $y$ to write 
\[
    \mathcal{F}_{y\to \eta}(-\Delta_Gu)(x;\eta)=(-\Delta_x + |x|^2|\eta|^2) \hat{u}(x;\eta),     
\]
where we have written $\hat{u}$ for $\mathcal{F}_{y\to\eta}(u)$. The operator $-\Delta_x + |x|^2|\eta|^2$ is a rescaled harmonic oscillator. Let $H = -\Delta_x + |x|^2$ be the usual harmonic oscillator acting on $\mathcal{S}(\mathbb{R}^{d_1})$. Since the spectrum of $H$ is pure-point, made of the eigenvalues $\lambda_m=2m+d_1$, $m\geqslant 0$, denote by $\{h_{m,k}\}_{k\in \Lambda_m}$ a basis of $L_x^2$-normalized eigenfunctions associated to the eigenvalue $\lambda_m$ (note that $\# \Lambda_m \sim m^{d_2-1}$). Then, we can see that for a fixed $|\eta|>0$, the eigenvalues of $\mathcal{F}_{y\to \eta}(-\Delta_G)(\cdot; \eta)$ are the $|\eta|(2m+d_1)$, $m\geqslant 0$ with $L^2_x$ normalized eigenfunctions given by  
\[
  \tilde{h}_{m,k}(x;\eta)=|\eta|^{\frac{d_1}{4}} h_{m,k}(|\eta|^{\frac{1}{2}}x), \quad k\in \Lambda_n. 
\]
Note that the factor $|\eta|^{\frac{d_1}{4}}$ stems for normalization in $L^2_x$.  

With $f_{m,k}(\eta) := (\hat{u}(\cdot, \eta),\tilde{h}_{m,k}(\cdot, \eta))_{L^2_x}$ we can write 
\[
    \hat{u}(x;\eta) = \sum_{\substack{m\geqslant 0 \\ k\in \Lambda_m}} f_{m,k}(\eta)\tilde{h}_{m,k} =: \sum_{\substack{m\geqslant 0 \\ k\in \Lambda_m}} \hat{u}_{m,k} = \sum_{m\geqslant 0} \hat{u}_m,     
\]
where $\hat{u}_m=\sum_{k\in\Lambda_m} \hat{u}_{m,k}$. Now $u_m:=\mathcal{F}^{-1}_{\eta \to y}(\hat{u}_m(x;\eta))$ so that  
\[
    u=\sum_{m \geqslant 0} u_m = \sum_{\substack{m\geqslant 0 \\ I \in 2^{\mathbb{Z}}}} u_{m,I} = \sum_{A \in 2^{\mathbb{N}}}u_A,     
\] 
where 
\[
    u_A=\sum_{\substack{m\geqslant 0 \\ A \leqslant 1+ (2m+d_1)I \leqslant 2A}} u_{m,I},
\]
with $u_{m,I}=\chi\left(\frac{|D_y|}{I}\right)u_m$ where $\chi$ is a nonnegative smooth bump function supported on $[\frac{1}{2},\frac{5}{2}]$ equal to $1$ on $[1,2]$ and such that $1=\sum_{I\in2^{\mathbb{Z}}}\chi\left(\frac{|\eta|}{I}\right)$ for all $\eta\in \mathbb{R}^{d_2}$. 

Because the $\left\{\tilde{h}_{m,k}\right\}_{\substack{m\geqslant 0 \\ k\in \Lambda_m}}$ are orthogonal in $L^2_x$ it follows that the $\{u_m\}_{m\geqslant 0}$ are orthogonal in $L^{2}_{x,y}$, so that in particular:  
\[
    \|u\|_{L_{x,y}^2}^2 = \sum_{m\geqslant 0} \|u_m\|^2_{L_{x,y}^2} \sim \sum_{\substack{m\geqslant 0 \\ I \in 2^{\mathbb{Z}}}}\|u_{m,I}\|^2_{L_{x,y}^2} \sim \sum_{A\in 2^{\mathbb{N}}} \|u_A\|^2_{L_{x,y}^2},     
\]
where we have used quasi-orthogonality of the $\{u_{m,I}\}_{I\in 2^{\mathbb{Z}}}$. 

Such a frequency decomposition is very useful for computing Sobolev norms, as 
\[
    \mathcal{F}_{y\to \eta}((-\Delta_G)^s u_{m,I}) = ((2m+d_1)|\eta|)^s \hat{u}_{m,I}.
\] 
If $|\eta| \sim I$ then $1+(2m+d_1)|\eta| \sim 1+(2m+1)I \sim A$ so that 
\begin{equation}
    \label{eq.Hs-freq}
    \|u\|_{H^s_G}^2 \sim \sum_{A \in 2^{\mathbb{N}}} A^s \|u_A\|_{L^2_{x,y}}^2. 
\end{equation}

\section{Strichartz estimates}\label{sec.stri} 

First, let us remark that \eqref{eq.grushin-strichartz-dual} follows from \eqref{eq.grushin-strichartz} by duality, and that similarly \eqref{eq.grushin-strichartz-inhomogeneous} follows using the so-called Christ--Kiselev lemma. Therefore, we only prove \eqref{eq.grushin-strichartz}, and in order to do it, we rely on the classical strategy by first obtaining frequency localized estimates. For now, let us work indifferently on $\mathbb{R}^{d_1} \times \mathbb{R}^{d_2}$ or $\mathbb{R}^{d_1} \times M$, and $\sigma \geqslant 1$. Fix some triple $(p_{\varepsilon},q,r)$ as in \cref{thm.main-Strichartz}. 

Using the triangle inequality on frequency decomposition and the Cauchy-Schwarz inequality we have: 
\begin{equation}
    \label{eq.CSapp}
    \|e^{it(-\Delta_G)^{\sigma}}u\|_{L^p_TL^q_xL^r_y} \leqslant \sum_{A\in 2^{\mathbb{N}}} \|e^{it(-\Delta_G)^{\sigma}}u_A\|_{L^p_TL^q_xL^r_y}\lesssim \sum_{A \in 2^{\mathbb{N}}} A^{\frac{\varepsilon}{2}}\|e^{it(-\Delta_G)^{\sigma}}u_A\|_{L^p_TL^q_xL^r_y}^2,     
\end{equation}
so that \eqref{eq.grushin-strichartz} is a consequence of 
\begin{equation}
    \label{eq.grushin-strichartz-A}
    \|e^{it(-\Delta_G)^{\sigma}}u_A\|_{L^p_TL^q_xL^r_y} \leqslant C A^{\frac{\gamma}{2}+\frac{\varepsilon}{4}}\|u\|_{L^{2}_{x,y}},    
\end{equation}
and \eqref{eq.Hs-freq}. 

The proof of \eqref{eq.grushin-strichartz} is itself a consequence of mode-wise Strichartz estimates. More precisely, we claim that \eqref{eq.grushin-strichartz} is a consequence of 
\begin{equation}
    \label{eq.grushin-strichartz-mode}
    \left\|e^{it((2m+d_1)|D_y|)^{\sigma}}\chi\left(\frac{|D_y|}{I}\right)u_m\right\|_{L^p_TL^q_xL^r_y} \leqslant C(A,m)\left\|\tilde{\chi}\left(\frac{|D_y|}{I}\right)u_m\right\|_{L^{2}_{x,y}},   
\end{equation}
where $I \in 2^{\mathbb{Z}}$ is such that $A \leqslant 1+(2m+d_1)I \leqslant 2A$, $\tilde{\chi}$ is a cutoff such that $\tilde{\chi}\chi = \tilde{\chi}$ and where 
\begin{equation}
    \label{eq.CA}
        C(A,m) = \begin{cases}
        \frac{A^{\frac{\gamma}{2}+\frac{\varepsilon}{4}}}{(m+1)^{d_2(\frac{1}{2}-\frac{1}{r})}} & \text{ in the Euclidean case,}\\ 
        \frac{A^{\frac{\gamma}{2} + \frac{\sigma-1}{p}+\frac{\varepsilon}{4}}}{(m+1)^{d_2(\frac{1}{2}-\frac{1}{r}) - \frac{1}{p}}} & \text{ in the compact manifold case.}
    \end{cases}  
\end{equation}
Since by assumption $d_2(\frac{1}{2}-\frac{1}{r})>\frac{1}{2}$ in the Euclidean case (resp. $d_2(\frac{1}{2}-\frac{1}{r}) - \frac{1}{p} > \frac{1}{2}$ in the compact manifold case), there holds $\sum_{m\geqslant 0} C(A,m)^2 \lesssim A^{\gamma + \frac{\varepsilon}{2}}$. Write 
\[
    u_A = \sum_{\substack{m\geqslant 0 \\ I: (m+1)I \sim A}} u_{m,I} =  \sum_{\substack{m\geqslant 0 \\ I: (m+1)I \sim A}} \tilde{\chi}\left(\frac{|D_y|}{I}\right)u_{m},    
\]
where $(m+1)I\sim A$ means $A \leqslant 1+(2m+d_1)I \leqslant 2A$, and that 
\[ 
    e^{it(-\Delta_G)^{\sigma}}\tilde{\chi}\left(\frac{|D_y|}{I}\right)u_{m} = e^{it((2m+1)|D_y|)^{\sigma}}\tilde{\chi}\left(\frac{|D_y|}{I}\right)u_{m},
\]
so that we can see that \eqref{eq.grushin-strichartz-A} follows from \eqref{eq.grushin-strichartz-mode} by an application of the triangle inequality and the Cauchy-Schwarz inequality.  

In the Euclidean case, \eqref{eq.grushin-strichartz-mode} will follow from \cref{prop.strichartz-modewise} applied to $\tilde{\chi}\left(\frac{|D_y|}{I}\right)u_m$ in place of $u_m$, and the fact that 
\[
    \left\|\langle\p_x \rangle ^{d_1(\frac{1}{2}-\frac{1}{q})}\tilde{\chi}\left(\frac{|D_y|}{I}\right)u_m\right\|_{L^{2}_{x,y}} \leqslant A^{\frac{d_1}{2}\left(\frac{1}{2}-\frac{1}{q}\right)} \|u_{m,I}\|_{L^{2}_{x,y}},
\]
which follows by applying \cref{lem.x-control}. 

Similarly, in the compact manifold case, \eqref{eq.grushin-strichartz-mode} will follow from \cref{prop.strichartz-modewise-torus} applied to $\tilde{\chi}\left(\frac{|D_y|}{I}\right)u_m$ in place of $u_m$, which yields \eqref{eq.grushin-strichartz-mode} on a time-interval $[0,\frac{c}{(m+1)A^{\sigma -1}}]$ with constant
\[
    C = \frac{1}{((m+1)A^{\sigma - 1})^{1/p}}C(A,m),  
\]
so that gluing the $L^p_t$ estimates on $\sim (m+1)A^{\sigma - 1}$ intervals yields \eqref{eq.grushin-strichartz-mode}. 

\subsection{Estimates on the Euclidean space}

The goal of this paragraph is to prove the following mode-wise Strichartz estimate. 

\begin{proposition}[Mode Strichartz estimate: Euclidean case]\label{prop.strichartz-modewise}
Let $m\geqslant 0$ and $I \in 2^{\mathbb{Z}}$ such that $A \leqslant 1+(2m+d_1)I \leqslant 2A$. 
Let $(p,q,r)$ be admissible (resp. $\sigma$-admissible). Then for all $u_m \in \mathcal{S}(\mathbb{R}_x^{d_1}\times\mathbb{R}_y^{d_2})$ there holds  
\begin{equation}
    \label{eq.grushin-strichartz-modewise}
    \left\|e^{it((2m+d_1)|D_y|)^{\sigma}}\chi\left(\frac{|D_y|}{I}\right)u_m\right\|_{L^p_TL^q_xL^r_y} \leqslant C'(A,m)\left\|\langle\p_x \rangle ^{(d_1+\varepsilon)(\frac{1}{2}-\frac{1}{q})} u_m\right\|_{L^{2}_{x,y}},   
\end{equation}
where we have set 
\begin{equation}
    \label{eq.def-CIm}
    C'(A,m)= \begin{cases}
        \left(\frac{A^{\frac{d_2+1}{2}}}{(m+1)^{d_2}} \right)^{\frac{1}{2}-\frac{1}{r}} & \text{ in the case } \sigma = 1, d_2 \geqslant 2, \\ 
        \left(\frac{A^{\frac{d_2}{2}(2-\sigma)}}{(m+1)^{d_2}} \right)^{\frac{1}{2}-\frac{1}{r}} & \text{ in the case } \sigma > 1, d_2 \geqslant 1.
    \end{cases}    
\end{equation}
and $\epsilon >0$ when $q=+\infty$ while we can take $\epsilon =0$ when $q<+\infty$.
\end{proposition}

The proof of \eqref{eq.grushin-strichartz-modewise} follows the classical strategy of obtaining dispersion estimates and using an abstract $TT^*$ argument to obtain the time-space estimate. The dispersive estimate relies on the dispersion of the fractional Schrödinger equation (when $d_2 \geqslant 1$) and the dispersion of the half-wave equation (when $d_2 \geqslant 2$). 
 
\begin{lemma}[Dispersion estimate] Let $d\geqslant 1$, $N\geqslant 1$ and $\chi$ be a smooth cutoff function supported on $[1,2]$. Let $v \in L^1(\mathbb{R}^d)$. 
\begin{enumerate}[label=(\roman*)]
    \item If $\sigma >1$, then for all $\tau > 0$ there holds    
    \begin{equation}
        \label{eq.disp-fractional}
        \left\|\exp\left(i\tau N\left(\frac{|D_y|}{N}\right)^{\sigma}\right)\chi\left(\frac{|D_y|}{N}\right)v\right\|_{L^{\infty}_y} \leqslant \frac{CN^{d}}{(1+N |\tau|)^{\frac{d}{2}}}\|v\|_{L^1_y}.  
    \end{equation}
    \item In case $\sigma = 1$ and $d \geqslant 2$, for all $\tau > 0$ there holds  
    \begin{equation}
        \label{eq.disp-half}
        \left\|\exp\left(i\tau N\frac{|D_y|}{N}\right)\chi\left(\frac{|D_y|}{N}\right)v\right\|_{L^{\infty}_y} \leqslant \frac{CN^{d}}{(1+N |\tau|)^{\frac{d-1}{2}}}\|v\|_{L^1_y}. 
    \end{equation} 
\end{enumerate}    
\end{lemma}

\begin{proof} The proof is standard, we recall it for the sake of completeness. Start by writing  
\[
    \exp\left(i\tau N\left(\frac{|\nabla|}{N}\right)^{\sigma}\right)\chi\left(\frac{|D_y|}{N}\right)v(y)= N^d \int_{y' \in \mathbb{R}^d} \Big(\underbrace{\int_{\eta \in \mathbb{R}^d}e^{iN((y-y')\cdot \eta + \tau|\eta|^{\sigma})}\chi(\eta)\mathrm{d}\eta}_{K_N(\tau,y,y')}\Big) v(y')\mathrm{d}y', 
\]
so that \eqref{eq.disp-fractional} and \eqref{eq.disp-half} follow from 
\begin{equation}
    \label{eq.kernel-fractional}
    \sup_{y,y'} K_N(\tau,y,y') \lesssim (1+|N\tau|)^{-\frac{d}{2}},
\end{equation}
\begin{equation}
    \label{eq.kernel-half}
    \sup_{y,y'} K_N(\tau,y,y') \lesssim (1+|N\tau|)^{-\frac{d-1}{2}}.
\end{equation}
Let us write $t=\tau N$ and remark that \eqref{eq.kernel-fractional} and \eqref{eq.kernel-half} are readily obtained in the case $0\leqslant t \leqslant C$ by compactness of the support of $\chi$. We can therefore assume $t\geqslant C>0$ in the following. Since the phase function in $K_N(\tau,y,y')$ only depends on $\tau, N$ and $Y:=N(y'-y)$ we write this phase function $\Phi(t,\eta,Y):=\frac{1}{t}Y\cdot \eta - |\eta|^{\sigma}$, so that we are left with showing that 
\[
    K(t,Y):=\int_{\eta \in \mathbb{R}^d}e^{-it\Phi(t,\eta,Y)} \chi(\eta)\mathrm{d}\eta
\]
decays as in \eqref{eq.kernel-fractional} and \eqref{eq.kernel-half}. 

We compute 
\begin{align*}
    \nabla_{\eta} \Phi(t,\eta,Y) &= \frac{1}{t}Y - \sigma|\eta|^{\sigma -2}\eta,\\
    \nabla^2_{\eta} \Phi(t,\eta,Y) &= -\sigma|\eta|^{\sigma - 2}\left(\operatorname{id} +(\sigma -2) \frac{\eta \cdot \eta^T}{|\eta|^2}\right),   \\ 
    |\det(\nabla^2_{\eta} \Phi(t,\eta,Y))| &= \sigma ^d (\sigma -1)|\eta|^{d(\sigma -2)}.
\end{align*}
Given the localization of $\eta = \mathcal{O}(1)$, the critical points solving $\nabla_{\eta}\Phi(t,\eta,Y)=0$ only exist when $\frac{|Y|}{t}=\mathcal{O}(1)$. Therefore, in the case $\frac{|Y|}{t} \geqslant C \gg 1$ there are no critical points. As in fact $|\nabla_{\eta}\Phi(t,\eta, Y)| \geqslant C$, the nonstationary phase lemma \cite[Chapter VIII, Proposition 4]{Stein} implies $|K(t,Y)| \leqslant C_kt^{-k}$ for all $k\geqslant 0$. In case $\frac{|Y|}{t}=\mathcal{O}(1)$ and if $\eta$ is a critical point, there holds $|\det(\nabla^2_{\eta} \Phi(t,\eta, Y))| \geqslant C >0$, so that the stationary phase lemma \cite[Theorem 7.7.6]{HormanderI} implies the claimed estimate. 

In the half-wave case we can also assume $|\eta| \geqslant C$, and note that critical points $\eta$ only exist if $|Y|=t$. Changing variables to polar coordinates we may write 
\[
    K(t,Y)=\int_{0}^{\infty}\int_{\omega \in \mathbb{S}^{d-1}}e^{ir(Y\cdot \omega - t)}\chi(r)r^{d-1}\mathrm{d}\sigma(\omega) \mathrm{d}r = \int_0^{\infty} e^{-irt} \hat{\sigma}(rY) \chi(r)r^{d-1} \mathrm{d}r,     
\]
where $\hat{\sigma}$ is the Fourier transform of the surface measure $\sigma$ of $\mathbb{S}^1$. In the case $t \geqslant C|Y|$, we have $|Y \cdot \omega - t| \geqslant \frac{Ct}{2}$ (at least when $C$ is large enough), so that repeated integrations by parts in $r$ yield $|K(t,Y)|\leqslant C_kt^{-k}$ for all $k\geqslant 0$. Finally, it remains to deal with the case $t \leqslant C |Y|$ and use the well-known stationary-phase type estimate $|\hat{\sigma}(\xi)| \leqslant |\xi|^{-\frac{d-1}{2}}$ \cite[Chapter VIII, eq. (15), (25)]{Stein} to bound 
\[
    |K(t,Y)| \leqslant C \int_0^{\infty}|rY|^{-\frac{d-1}{2}} \chi(r)r^{d-1}dr \lesssim  Y^{-\frac{d-1}{2}} \lesssim t^{-\frac{d-1}{2}},  
\]
which provides the claimed estimates.
\end{proof}

Equipped with such dispersion estimates, we can derive mode-wise dispersion estimates for frequency localized data. 

\begin{corollary}[Dispersion on Hermite mode $m$]\label{coro.disp-mode-m}
Let $m\geqslant 0$, $I\in 2^{\mathbb{Z}}$ and $A \in 2^{\mathbb{N}}$ such that $A \leqslant 1 + (2m +d_1)I \leqslant 2A$. Then we have for all $t \geqslant 0$, and all $(q,r) \in [2,\infty]^2$, 
\begin{equation}
    \label{eq.dispersion-half-m}
    \left\|e^{it(2m+d_1)|D_y|}\chi\left(\frac{|D_y|}{I}\right)\langle \partial_x \rangle ^{-(d_1+\varepsilon)\left(1-\frac{2}{q}\right)}u_m\right\|_{L^q_xL^{r}_y} \leqslant C\left(\frac{A^{\frac{d_2+1}{2}}}{(m+1)^{d_2}t^{\frac{d_2-1}{2}}}\right)^{\frac{1}{r'}-\frac{1}{r}}\|u_m\|_{L^{q'}_xL^{r'}_y},   
\end{equation}
and similarly if $\sigma \neq 1$, 
\begin{equation}
    \label{eq.dispersion-fract-m}
    \left\|e^{it((2m+d_1)|D_y|)^{\sigma}}\chi\left(\frac{|D_y|}{I}\right)\langle \partial_x \rangle ^{-(d_1+\varepsilon)\left(1-\frac{2}{q}\right)}u_m\right\|_{L^q_xL^{r}_y} \leqslant C\left(\frac{A^{d_2(1-\frac{\sigma}{2})}}{(m+1)^{d_2}t^{\frac{d_2}{2}}}\right)^{\frac{1}{r'}-\frac{1}{r}}\|u_m\|_{L^{q'}_xL^{r'}_y}.  
\end{equation}
In both cases $C$ depends only on $(q,r,d_1,d_2)$. 
\end{corollary}

\begin{proof} As both claims are similar, we only explain how to obtain the half-wave case. By applying \eqref{eq.disp-half} to $v=v_m:=\langle \partial_x \rangle ^{-(d_1+\varepsilon)\left(1-\frac{2}{q}\right)}u_m(x,\cdot)$ with $N=I$, $d=d_2$ and $\tau = (2m+d_1)t$, we get 
\begin{multline*}
    \left\|e^{it(2m+d_1)|D_y|}\chi\left(\frac{|D_y|}{I}\right)v_m\right\|_{L^{\infty}_y} \leqslant \frac{C A^{\frac{d_2+1}{2}}}{(2m+d_1)^{d_2}(1+ t)^{\frac{d_2-1}{2}}}\|v_m(x,\cdot)\|_{L^{1}_y} \\
    \leqslant \frac{CA^{\frac{d_2+1}{2}}}{(m+1)^{d_2}t^{\frac{d_2-1}{2}}}\|v_m(x,\cdot)\|_{L^{1}_y},  
\end{multline*}
In view of
\[
    \left\|e^{it(2m+d_1)|D_y|}\chi\left(\frac{|D_y|}{I}\right)v_m\right\|_{L^{2}_y} \leqslant \|v_m\|_{L^{2}_y}, 
\]
we can use complex interpolation to obtain that for all $r \in [2,\infty]$ there holds 
\begin{multline}
    \label{eq.norm-y}
    \left\|e^{it(2m+d_1)|D_y|}\chi\left(\frac{|D_y|}{I}\right)\langle \partial_x \rangle ^{-(d_1+\varepsilon)(1-2/q)}u_m\right\|_{L^{r}_y} \leqslant C\left(\frac{A^{\frac{d_2+1}{2}}}{(m+1)^{d_2}t^{\frac{d_2-1}{2}}}\right)^{\frac{1}{r'}-\frac{1}{r}} \\
    \times \|\langle \partial_x \rangle ^{-(d_1+\varepsilon)(1-2/q)}u_m\|_{L^{r'}_y}.  
\end{multline}
Taking the $L^2_x$ norms in \eqref{eq.norm-y}, one obtains \eqref{eq.dispersion-half-m} in the case $(q,r)=(2,r)$.  

If $q=\infty$, we also take the $L^{\infty}_x$ norm in \eqref{eq.norm-y}:
\begin{multline*}
    \left\|e^{it(2m+d_1)|D_y|}\chi\left(\frac{|D_y|}{I}\right)\langle \partial_x \rangle ^{-(d_1+\varepsilon)}u_m\right\|_{L^{\infty}_xL^{r}_y} \leqslant C\left(\frac{A^{\frac{d_2+1}{2}}}{(m+1)^{d_2}t^{\frac{d_2-1}{2}}}\right)^{\frac{1}{r'}-\frac{1}{r}} \\ 
    \times \|\langle \partial_x \rangle ^{-(d_1+\varepsilon)}u_m\|_{L^{\infty}_xL^{r'}_y}.
\end{multline*} 
Using the triangle inequality, the Sobolev embedding and the Minkowski inequality we bound
\[
    \|\langle \partial_x \rangle ^{-(d_1+\varepsilon)}u_m\|_{L^{\infty}_xL^{r'}_y} \leqslant \|\langle \partial_x \rangle ^{-(d_1+\varepsilon)}u_m\|_{L^{r'}_yL^{\infty}_x} \leqslant C \|u_m\|_{L^{r'}_yL^{1}_x} \leqslant  C \|u_m\|_{L^{1}_xL^{r'}_y}.
\]
This yields
\[
    \left\|e^{it(2m+d_1)|D_y|}\chi\left(\frac{|D_y|}{I}\right)\langle \partial_x \rangle ^{-(d_1+\varepsilon)}u_m\right\|_{L^{\infty}_xL^{r}_y} \leqslant C\left(\frac{A^{\frac{d_2+1}{2}}}{(m+1)^{d_2}t^{\frac{d_2-1}{2}}}\right)^{\frac{1}{r'}-\frac{1}{r}}\|u_m\|_{L^{1}_xL^{r'}_y}.      
\]
We then obtain \eqref{eq.dispersion-half-m} using Stein's complex interpolation theorem \cite{S56} between $(q,r)=(2,r)$ and $(q,r)=(\infty,r)$.
\end{proof}

\begin{proof}[Proof of \cref{prop.strichartz-modewise}]
We only detail the half-wave case $\sigma = 1$, as the fractional case $\sigma \neq 1$ is similar. In order to prove \eqref{eq.grushin-strichartz-modewise} we use the $TT^*$ argument.  Let $T_m$ be the operator defined for all $v \in L^2_{x,y}$ by 
\[ 
    T_mv(t)=e^{it(2m+d_1)|D_y|}\chi\left(\frac{|D_y|}{I}\right)\langle \partial_x\rangle^{-(d_1+\varepsilon)\left(\frac{1}{2}-\frac{1}{q}\right)}v.
\] 
Now \eqref{eq.grushin-strichartz-modewise} amounts to proving the continuity bound  
\[ 
    \|T_m\|_{L^2_{x,y} \to L^p_TL^q_xL^r_y} \leqslant \frac{CA^{\frac{d_2+1}{2}(\frac{1}{2}-\frac{1}{r})}}{(m+1)^{d_2(\frac{1}{2}-\frac{1}{r})}} 
\]
which is equivalent to showing continuity
\[
    \|T_mT_m^*\|_{L^{p'}_TL^{q'}_{x}L^{r'}_y \to L^p_TL^q_xL^r_y} \leqslant \frac{CA^{(d_2+1)(\frac{1}{2}-\frac{1}{r})}}{(m+1)^{2d_2(\frac{1}{2}-\frac{1}{r})}} 
\]
One can compute 
\begin{equation*}
    \label{eq.TTm}
    T_mT_m^*v(t)=\int_{\mathbb{R}} e^{i(t-t')(2m+d_1)|D_y|}\chi\left(\frac{(2m+d_1)|D_y|}{A}\right)\langle \partial_x\rangle^{-(d_1+\varepsilon)(1-\frac{2}{q})}v(t')\mathrm{d}t', 
\end{equation*}
where we have used that $\langle \partial_x \rangle$ commutes with $e^{it(2m+d_1)|D_y|}\chi\left(\frac{|D_y|}{I}\right)$. 

Using the triangle inequality and the dispersive estimate \eqref{eq.dispersion-half-m} yields 
\[
    \|T_mT_m^*v(t)\|_{L^q_xL^r_y} \lesssim \int_{\mathbb{R}} \left(\frac{A^{\frac{d_2+1}{2}}}{(m+1)^{d_2}|t-t'|^{\frac{d_2-1}{2}}}\right)^{\frac{1}{r'}-\frac{1}{r}} \|v(t')\|_{L^{q'}_xL^{r'}_y}\mathrm{d}t',  
\]
Next, we apply the Hardy-Littlewood-Sobolev inequality, which requires that $\frac{d_2-1}{2}(\frac{1}{r'}-\frac{1}{r})<1$ (in case $\sigma >1$ it would be $\frac{d_2}{2}(\frac{1}{r'}-\frac{1}{r})<1$). Therefore, we arrive at  
\[
    \|T_mT_m^*v(t)\|_{L^p_TL^q_xL^r_y} \lesssim \frac{A^{\frac{d_2+1}{2}(1-\frac{2}{r})}}{(m+1)^{d_2(1-\frac{2}{r})}} \|v\|_{L^{p'}_TL^{q'}_xL^{r'}_y},
\] 
provided $(p,r)$ are such that 
\[
    1 + \frac{1}{p} = \frac{1}{p'} + \frac{d_2-1}{2}\left(1-\frac{2}{r}\right),    
\]
which rewrites as $\frac{2}{p} = (d_2-1)\left(\frac{1}{2} - \frac{1}{r}\right)$, which is exactly the scaling condition \eqref{eq.adm-S} and ends the proof. 
\end{proof}

\subsection{Estimates on $\mathbb{R}^{d_1}\times M$} 

In this section $(M,g)$ is a smooth compact manifold without boundary. 
Our goal is to obtain the following mode-wise frequency localized Strichartz estimate, which is essentially a consequence of \cite{dinh}, itself an adaptation of the analysis of \cite{BGT} to the case of fractional Schrödinger operators. 

\begin{proposition}[Mode Strichartz estimate: compact manifold case]\label{prop.strichartz-modewise-torus} Let $\sigma \in [1,\infty)$.
    Let $m\geqslant 0$ and $I \in 2^{\mathbb{Z}}$ such that $A \leqslant 1+(2m+d_1)I \leqslant 2A$. 
    Let $(p,q,r)$ be admissible (or $\sigma$-admissible). Then for all smooth functions $u_m$ defined on $\mathbb{R}^{d_1}\times M$, there holds  
    \begin{equation}
        \label{eq.grushin-strichartz-modewise-torus}
        \left\|e^{it((2m+d_1)|D_y|)^{\sigma}}\chi\left(\frac{|D_y|}{I}\right)u_m\right\|_{L^p_TL^q_xL^r_y} \leqslant C'(A,m)\left\|\langle\p_x \rangle ^{(d_1+\varepsilon)(\frac{1}{2}-\frac{1}{q})}u_m\right\|_{L^{2}_{x,y}},   
    \end{equation}
    where $T = \frac{C}{(m+1)A^{\sigma -1}}$, and where $C'(A,m)$ is defined by \eqref{eq.def-CIm}.
\end{proposition}

\begin{remark} Estimates of the form \eqref{eq.grushin-strichartz-modewise-torus} can be inferred from the Strichartz estimates that are best suited for the problem. In this article we work on a general compact manifold without boundary, and we have used the set of estimates from~\cite{dinh}.  However, in some particular cases, better estimates might be available and this would improve in turn our results.  Let us illustrate how the situation may differ on the specific case of the torus $M=\mathbb{T}^2$. Let us also choose for simplicity $(d_1,d_2,\sigma)=(1,2,2)$.     
On the one hand, gluing \eqref{eq.grushin-strichartz-modewise-torus} on $\sim (m+1)A$ intervals gives 
\begin{multline}
    \label{eq.general-manifold}
    \left\|e^{it(2m+2)^2\Delta_y}\chi\left(\frac{|D_y|}{I}\right)u_m\right\|_{L_{t,x,y}^{4}([0,1] \times \mathbb{R}\times \mathbb{T}^2)} \lesssim \frac{A^{\frac{1}{4}}}{(m+1)^{\frac{1}{2}}}\left\|\langle\p_x \rangle ^{\frac{1+\varepsilon}{4}}u_m\right\|_{L^{2}_{x,y}}\\
     \lesssim \frac{A^{\frac{3}{8} + \varepsilon}}{(m+1)^{\frac{1}{4}}}\|u_m\|_{L^2_{x,y}}.   
\end{multline}
On the other hand, using the Strichartz estimate $L^2_y(\mathbb{T}^2) \to L^4_{t,y}([0,1]\times \mathbb{T}^2)$ \cite{zygmund} and rescaling $t$ to $(2m+2)^2t$ as well as the Sobolev embedding $H^{\frac{1}{4}}_x \hookrightarrow L^4_x$ yields 
\[
    \left\|e^{it(2m+2)^2\Delta_y}\chi\left(\frac{|D_y|}{I}\right)u_m\right\|_{L_{t,x,y}^{4}([0,1] \times \mathbb{R}\times \mathbb{T}^2)}\lesssim A^{\frac{1}{8}}\|u_m\|_{L^2_{x,y}}
\] 
which avoids an $A^{\frac{1}{4}}$ loss compared to \eqref{eq.general-manifold}. 

This is not surprising as \eqref{eq.grushin-strichartz-modewise-torus} ultimately relies on Strichartz estimates on general compact domains without boundaries, for which in general Strichartz estimates are known to be worse than for the torus. 
\end{remark}

\begin{proof} In the case $\sigma \neq 1$ we can use \cite[Proposition 3.5, (3.8)]{dinh}. With $(h^{-1},t,\varphi,u_0)=(N,\tau,\chi,\tilde{\chi}v)$ where $\tilde{\chi}$ satisfies $\chi\tilde{\chi}=\tilde{\chi}$, this writes: 
\begin{equation}
    \label{eq.disp-compactM}
    \left\|\exp\left(i\tau N\left(\frac{|D_y|}{N}\right)^{\sigma}\right)\chi\left(\frac{|D_y|}{N}\right)v\right\|_{L^{\infty}_y} \leqslant \frac{CN^{d_2}}{(1+N |\tau|)^{\frac{d_2}{2}}}\left\|\tilde{\chi}\left(\frac{|D_y|}{N}\right)v\right\|_{L^1_y},   
\end{equation} 
and holds for $|\tau| \leqslant c$ for some $c>0$. The proof of \eqref{eq.grushin-strichartz-modewise-torus} can now be carried out by reproducing that of \cref{coro.disp-mode-m}.

The case $\sigma=1$ is even simpler by using exact finite speed of propagation. To that end, let us fix $u_0 \in L^1_y(M)$ and let us also choose coordinate charts $\{U_j,V_j,\theta_j\}_{1 \leqslant j \leqslant J}$ of $M$, where $\theta_i : U_i \subset M \to V_i \subset \mathbb{R}^d$ are diffeomorphisms. Let us also fix $1=\sum_{1 \leqslant j\leqslant J} \psi_j$, a partition of unity of $M$ adapted to such coordinate charts. Because the partition of unity is finite, we can restrict to initial data of the form $v_0=\tilde{\psi}\chi\left(\frac{|D_y|}{I}\right)u_0$, where $\tilde{\psi}$ is such that $\operatorname{supp}(\tilde{\psi})$ is compactly contained in $\{\psi = 1\}$ for some $\psi = \psi_j$. Omitting the subscript $j$, we let $(U,V,\theta)$ be the corresponding patch. In this patch, the operator $|D_y|=\sqrt{-\Delta_y}$ reads $\sqrt{-\Delta_g}$ where $-\Delta_g$ is the Laplace--Beltrami operator associated to $(M,g)$. 

Consider $v_F(\tau)=S(\tau)(\theta_* v_0)$ where $S(\tau)=\exp\left(i \tau \sqrt{-\Delta_g}\right)$ is the free evolution on $\mathbb{R}^{d_2}$ and where $\theta_*v_0 := v_0 \circ \theta^{-1}$. Because $v_F$ is a solution to the half-wave equation, hence of the wave equation on $\mathbb{R}^d$, it enjoys finite speed propagation at speed $1$. It follows that there exists $c>0$ such that $\operatorname{supp}(\theta^*S(\tau)\theta_* v_0) \subset \operatorname{supp}(\psi)$ for all $\tau \in [0,c]$. 
    
It follows that the solution $v(t)=\exp\left(iN\tau \left(\frac{|D_y|}{N}\right)^{\sigma}\right)v_0$ coincides with $\theta^*S(\tau)\theta_* v_0$ for all $\tau \in [0,c]$, from which we can conclude that \eqref{eq.disp-half} is also valid in the compact manifold case, on a $\mathcal{O}(1)$ timescale.
\end{proof}

\section{Applications to the Cauchy theory of nonlinear dispersive equations}\label{sec.cauchy}

In order to prove \cref{thm.main-Cauchy}, let us fix $u_0 \in H^s_G$ and observe that $u$ solves \eqref{eq.NLS-G} (or \eqref{eq.fNLS-G}) in $X^s_T$ if and only if it is a fixed point of the functional $\Phi$ defined by 
\begin{equation}
    \label{eq.integral-NLSG}
    \Phi (u) (t)=e^{it(-\Delta_G)^{\sigma}}u_0 - i \int_0^t e^{i(t-t')(-\Delta_G)^{\sigma}}F(u(t'))\,\mathrm{d}t', 
\end{equation}
where $F(u)=|u|^{\kappa -1}u$.

The proof of \cref{thm.main-Cauchy} is now a consequence of the Banach fixed point theorem once we prove the following estimates. 

\begin{lemma}\label{lem.contraction-NLSG}
If $(\kappa,\sigma) = (5,1)$, assume $s>2$. If $\kappa = 3$ and $\sigma \in (1,\frac{5}{2})$ assume $s>\frac{5}{2}-\sigma$ in the Euclidean case and $s>2$ in the compact manifold case.

Then, there exists $\theta >0$ such that for any $u, v \in X^s_T$ there holds 
\begin{equation}
    \label{eq.stable-Xs}
    \|\Phi(u)\|_{X^s_T} \lesssim \|u_0\|_{H^s_G} + T^{\theta}\|u\|^{\kappa}_{X^s_T}, 
\end{equation}
as well as the contraction estimate: 
\begin{equation}
    \label{eq.contraction-Xs}
    \|\Phi(u)-\Phi(v)\|_{X^s_T} \lesssim T^{\theta}(\|u\|_{X^s_T}^{\kappa - 1} + \|v\|_{X^s_T}^{\kappa - 1})\|u-v\|_{X^s_T}.
\end{equation}
\end{lemma}

\begin{proof} 
Let us only prove \eqref{eq.stable-Xs}, as \eqref{eq.contraction-Xs} follows by similar methods. In the proof we will choose the triple $(p,q,r)$ appearing in the definition of the $X^s_T$ norm, for now we only require that $(p,q,r)$ is admissible (resp. $\sigma$-admissible).

First, using the Strichartz estimates \eqref{eq.grushin-strichartz} and \eqref{eq.grushin-strichartz-inhomogeneous}, choosing $(a,b,c)=(\infty, 2,2)$ (which is also an allowed triple, in this case the Strichartz estimate is straightforward) therefore $\gamma(a,b,c)=0$ and $(a',b',c')=(1,2,2)$ so that
\[
    \|\Phi(u)\|_{X^s(I)} \leqslant \|u_0\|_{H^s} + \|\langle \nabla \rangle ^s F(u)\|_{L^1_TL^{2}_{x,y}}, 
\]
where $(1,2,2)$ is the dual triple of the admissible triple $(\infty, 2,2)$. Then we use a nonlinear estimate adapted to the Baouendi--Grushin operator, see for instance \cite[Proposition 2.10]{LG22}, which yields 
\[
    \|\langle \nabla \rangle ^s F(u)\|_{L^1_TL^{2}_{x,y}} \lesssim \|u\|_{L^{\infty}_TH^s_G}\|u\|_{L^{\kappa -1}_TL^{\infty}_{x,y}}^{\kappa -1}. 
\]
It remains to prove that 
\begin{equation}
    \label{eq.bound-apriori}
    \|u\|_{L^{\kappa -1}_TL^{\infty}_{x,y}}^{\kappa -1} \lesssim T^{\theta}\|u\|_{X^s}^{\kappa -1}.  
\end{equation}
\begin{itemize}
    \item Assume $(\kappa, \sigma) = (5,1)$, $\sigma = 1$ and $s>2$. Let $\varepsilon \ll 1$. Let $(p_{\varepsilon},r_{\varepsilon})$ chosen such that $\frac{1}{p_{\varepsilon}} = \frac{1}{4} - \frac{\varepsilon}{2}$ and $\frac{1}{r_{\varepsilon}} = \varepsilon$. One can compute $\gamma_{\varepsilon}:=\gamma_{r_{\varepsilon},r_{\varepsilon}}=2-4\varepsilon$ and that $(p_{\varepsilon},r_{\varepsilon},r_{\varepsilon})$ is admissible. An application of the Sobolev embedding \eqref{eq.sobolev-Wkp} and the Hölder inequality in $t$ yields 
    \[
        \|u\|_{L^{4}_TL^{\infty}_{x,y}} \lesssim \|\langle \Delta _G \rangle ^{\frac{5\varepsilon}{2}} u\|_{L^4_TL_{x,y}^{r_{\varepsilon}}} \lesssim T^{\theta(\varepsilon)} \|\langle \Delta _G \rangle ^{\frac{5\varepsilon}{2}} u\|_{L^{p_\varepsilon}_TL_{x,y}^{r_{\varepsilon}}}, 
    \] 
    with $\theta_{\varepsilon} = \frac{\varepsilon}{2}>0$. In order to conclude, observe that $\gamma_{\varepsilon} = 2 - 4\varepsilon < s - 5\varepsilon$ for sufficiently small $\varepsilon$, and \eqref{eq.bound-apriori} follows. The compact case works the same, as $(p_\varepsilon,r_{\varepsilon})$ satisfies \eqref{eq.adm-S3}.
    \item Assume $\kappa = 3$, $\sigma \in (1,2)$ and $s>\frac{5}{2}-\sigma$. A similar application of the Sobolev embedding \eqref{eq.sobolev-Wkp} and the Hölder inequality in $t$, with $\frac{1}{p_{\varepsilon}}=\frac{1}{2}-\varepsilon$ and $\frac{1}{r_{\varepsilon}}=\varepsilon$ yields  
    \[
        \|u\|_{L^{2}_TL^{\infty}_{x,y}} \lesssim T^{\theta(\varepsilon)} \|\langle \Delta _G \rangle ^{\frac{5\varepsilon}{2}} u\|_{L^{p_\varepsilon}_TL_{x,y}^{r_{\varepsilon}}},
    \]
    where $\theta(\varepsilon)=\varepsilon >0$. The triple $(p_{\varepsilon},r_{\varepsilon},r_{\varepsilon})$ is $\sigma$-admissible and $\gamma_{\varepsilon}:=\gamma_{r_{\varepsilon},r_{\varepsilon}} = (5-2\sigma)(\frac{1}{2}-\varepsilon)$. The conclusion follows for $\varepsilon$ small enough, since $s>5\varepsilon + \gamma_{\varepsilon}$. \qedhere
\end{itemize}
\end{proof}

\begin{remark} Note that in the compact case, $\kappa = 3$ and $\sigma \neq 1$, we would not be able to make the above method successful: the condition $\frac{1}{2}-\frac{1}{r_\varepsilon} > \frac{1}{4} + \frac{1}{4p_{\varepsilon}}$ from \eqref{eq.adm-fS3} would never be satisfied. 
\end{remark}

The global well-posedness statement follows from the local theory and conservation of the energy the $L^2$ norm of $u$ and the energy $E(u(t))=\frac{1}{2}\|(-\Delta_G)^{\frac{\sigma}{2}}u(t)\|^2_{L^2} + \frac{1}{4}\|u\|_{L^4}^4$, which controls $\|u(t)\|^2_{H_G^{\sigma}}$ and implies global existence.

\end{document}